\documentclass{amsart}
\pdfoutput=1
\usepackage[longtable]{multirow}
\usepackage{graphicx}
\usepackage{longtable}
\usepackage{tikz-cd}
\usepackage{chngcntr}
\usepackage{fullpage}
\usepackage{amsmath}
\usepackage{amsthm}
\usepackage{amsfonts}
\usepackage{mathtools}
\usepackage{caption}

\usepackage{amssymb}
\theoremstyle{theorem}
\newtheorem{theorem}{Theorem}
\newtheorem{cor}{Corollary}

\newtheorem{prop}{Proposition}

\newtheorem*{definition}{Definition}
\theoremstyle{definition}
   \newtheorem*{remark}{Remark}

\renewcommand{\arraystretch}{1.2}

\usepackage{rotating}
\usepackage{accents}
\title{Periods of singular double octic Calabi-Yau threefolds\\and 
modular forms}


\author{Tymoteusz Chmiel}
\address{Institute of Mathematics, Jagiellonian University,
ul. {\L}ojasiewicza 6,
30-348 Krak\'ow,
Poland}
\email{tymoteusz.chmiel@doctoral.uj.edu.pl}
\author{S\l{}awomir Cynk}
\address{Institute of Mathematics, Jagiellonian University,
	ul. {\L}ojasiewicza 6,
	30-348 Krak\'ow,
	Poland}
\email{slawomir.cynk@uj.edu.pl}
\thanks{The second author was partially supported by the National Science Center grant no.  2020/39/B/ST1/03358. }

\begin{document}
\maketitle

\begin{abstract}
By the modularity theorem every rigid Calabi-Yau threefold $X$ has associated modular form $f$ such that the equality of $L$-functions $L(X,s)=L(f,s)$ holds. In this case period integrals of $X$ are expected to be expressible in terms of the special values $L(f,1)$ and $L(f,2)$. We propose a similar interpretation of period integrals of a nodal model of $X$. It is given in terms of certain variants of a Mellin transform of $f$. We provide numerical evidence towards this interpretation based on a case of double octics.
\end{abstract}

\section{Introduction}\label{sec:1}

A \textit{Calabi-Yau threefold} is a smooth complex projective variety of 
dimension $3$ such that \[\Omega^{3}_X\simeq\mathcal{O}_X\text{ and } H^1(X,\mathcal{O}_X)=0.\] 
In particular, there exists a non-vanishing holomorphic $3$-form $\omega$ on $X$, unique up to a constant. \textit{Period integrals} of a Calabi-Yau threefold $X$ are integrals of $3$-form $\omega$ over integral $3$-cycles. We shall denote by $\Lambda _{X}$ the period lattice of $X$, i.e.:
\begin{eqnarray*}
&&\Lambda_X:=\left\{\int_{\gamma}\omega:\;\gamma\in H_{3}(X,\mathbb{Z})\right\} \subset \mathbb{C}
\end{eqnarray*}
As the canonical form $\omega$ is defined up to a constant factor, the lattice $\Lambda_X$ is also defined only up to rescaling.
If the Calabi-Yau threefold $X$ is rigid, then $\Lambda _{X}$ 
defines an elliptic curve 
\[J^{2}(X):=H^{3,0}(X)^{*}/H_{3}(X,\mathbb Z)=\mathbb C/\Lambda_X\]
which is a particular case of the Griffiths \emph{intermediate Jacobian} \cite{Griff}.

Period integrals were computed only for a small number of Calabi-Yau threefolds. In \cite{Cynk-van Straten}  approximations of period integrals of eleven double octics were established by a numeric integration 
using a very explicit description of the geometry of these varieties. These computations give strong numerical evidence of the proportionality between period integrals of a rigid Calabi-Yau threefold and special values of the $L$-function of the corresponding cusp form as predicted by the Tate conjecture.

In \cite{Chmiel} we proposed a different approach for computing period integrals. If a rigid Calabi-Yau threefold $X$ is a resolution of singularities of a singular element $X_{t_0}$ of a one-parameter family $\mathcal X=(X_{t})_{t\in\mathbb{C}}$, then we can compute periods of $X$ as limits of certain period integrals of smooth elements 
of the family $\mathcal X$. Since period integrals of Calabi-Yau threefolds $X_{t}$ satisfy the Picard-Fuchs equation, we can use this differential equation, especially its monodromy, to determine periods of $X$.
This approach has two important advantages. Firstly, it enables computation of much better approximations of periods. Secondly, it depends only on the Picard-Fuchs operator and does not require any knowledge of the geometry of the considered variety.

In fact the approach based on the Picard-Fuchs operator computes periods of a singular variety $X_{t_0}$ rather than only those of the rigid Calabi-Yau threefold $X$. In general, the group of period integrals of $X_{t_0}$ has rank 3, while the group of periods of $X$ has rank 2. In \cite{Chmiel} we verified that the additional periods computed using the monodromy of the Picard-Fuchs equation agree with additional integrals computed in \cite{Cynk-van Straten} for polyhedral cycles in the nodal model $X_{t_0}$ that do not lift to a cycle in $X$.

The main goal of this paper is to propose an interpretation of the period integrals of the singular fiber $X_{t_0}$. As showed in \cite{Cynk-van Straten}, period  integrals of its smooth, birational model $X$ are proportional to the special values of the $L$-function of the modular form $f$ associated with the rigid Calabi-Yau threefold $X$ by the modularity theorem. We provide numerical evidence for a similar proportionality between periods of $X_{t_0}$ and certain partial integrals $M(f,k)$, which appear naturally in the classical proof of the functional equation for $L(f,s)$. As a consequence, we get strong evidence that 
the period integrals of $X_{t_0}$ are also determined by the cusp form attached to $X$.

If a rigid Calabi-Yau threefold $X$ has models $X_1,X_2$ isomorphic over $\mathbb C$ but not over $\mathbb Q$, then the cusp forms $f_1,f_2$ attached to these models can also differ. To remedy this ambiguity of the cusp form, we usually consider the twist of a minimal level. However, the Picard-Fuchs operators of $\mathbb{C}$-isomorphic varieties coincide and so they do not distinguish between different models. Since for different twists of a fixed modular form $f$ the integrals $M(f_1,k),\ M(f_2,k)$ demonstrate no proportionality (at least on a numerical base with high accuracy), we believe that the period integrals of the singular variety $X_{t_0}$ determine a 'preferred' modular form in a more canonical way.

The paper is organized as follows. In section \ref{sec:CY} we introduce basic definitions concerning Calabi-Yau threefolds and their period integrals. We also describe double octics which are our main source of examples. In section \ref{sec:one-p} we proceed to the case of one-dimensional families and associated differential operators. Section \ref{sec:twists} provides necessary information on twists of a modular form $f$ and in section \ref{sec:M(f,k)} we define the partial integrals $M(f,s)$ associated with $f$. They are supposed to provide a tool for understanding integrals of singular double octic Calabi-Yau threefolds. Finally, in section \ref{sec:interpretation} we present numerical evidence towards this connection.

\section{Periods of rigid double octic Calabi-Yau threefolds}\label{sec:CY}

The Bogomolov-Tian-Todorov unobstructedness theorem implies that the universal deformation space of a Calabi-Yau threefold $X$ is a smooth manifold of dimension $h^{2,1}(X)$.  In particular, a Calabi-Yau threefold $X$ is \textit{rigid}, i.e. admits no deformations of the complex structure, exactly when $h^{2,1}(X)=0$. 
Similarly, a Calabi-Yau threefold has one-parameter universal deformation 
space if and only if $h^{2,1}(X)=1$.

By the Hodge decomposition for all Calabi-Yau manifolds $X$ we have the following equality:
\[b_{3}(X)=2h^{2,1}(X)+2\]
Consequently, for a rigid Calabi-Yau threefold 
the group $H_{3}(X,\mathbb Z)$ has rank 2. Fixing a (non-zero) canonical form $\omega\in H^{3,0}(X)$ we define the \emph{period lattice} of $X$ to be
\[\Lambda _{X}:=\left\{\int_{\gamma}\omega: \gamma\in H_{3}(X,\mathbb Z)\right\}.\]

Our paper was motivated by a phenomenon exhibited by numerical computations of period integrals of rigid double octic Calabi-Yau threefolds. 
A \textit{double octic} is a Calabi-Yau threefold obtained as a resolution of singularities of a double cover of $\mathbb{P}^3$ branched along a union of eight planes $D\subset\mathbb P^{3}$. Double octics defined over $\mathbb{Q}$ with the Hodge number $h^{2,1}\leq1$ were completely classified in \cite{Cynk-Kocel}.
Among them there are eleven rigid double octics defined over $\mathbb Q$.

Every rigid Calabi-Yau threefold defined over $\mathbb Q$ is modular. More  precisely the following modularity theorem from \cite{Goueva-Yui} is a consequence of the Serre Conjecture proven by Khare and Wintenberger:
\begin{theorem}\label{th:mod}
	Let $X$ be a rigid Calabi-Yau threefold defined over $\mathbb{Q}$. Then there exists an integer $N$ and a Hecke eigenform  $f\in\mathcal{S}_4(\Gamma_0(N))$ such that $L(X,s)=L(f,s)$.
\end{theorem}	
\noindent The level $N$ of the eigenform $f$ equals the product $\prod_{p}p^{e(p)}$ taken over the set of bad primes with $e(2)\le8$, $e(3)\le5$ and $e(p)\le2$ for $p\ge5$.
Modular forms for rigid double octics have been computed in \cite{Meyer}. Throughout this paper we shall use the numbering of double octics and modular forms introduced in this book. 

The modular form $f$ corresponding to a rigid Calabi-Yau threefold $X$ can be seen as a $2$-form on the associated Kuga-Sato variety $Y$, and the special values $L(f,1)$ and $\tfrac{L(f,2)}{2\pi i}$ as its periods. By the Tate conjecture the 
equality of $L$-functions $L(X,s)=L(f,s)$ should imply the existence of a correspondence between $X$ and $Y$. Consequently, we expect the lattices $\Lambda_X$ and \[\Lambda_f:=(2\pi i)^2L(f,1)\mathbb{Z}\oplus(2\pi i)L(f,2)\mathbb{Z}\] to be commensurable.

Numerical approximations of certain sublattices of  $\Lambda_X$ for rigid double octics were first computed in \cite{Cynk-van Straten}.
If a double octic $X$ is defined as a resolution of singularities $\sigma:X\longrightarrow \overline{X}$ of a double covering $\pi:\overline{X}\longrightarrow\mathbb P^{3}$ branched along a union of eight planes $D=P_{1}\cup\dots\cup P_{8}$, then planes $P_{i}$ define a decomposition of $\mathbb P^{3}(\mathbb R)$ into polyhedra. The double cover $\tilde C:=\pi^{-1}(C)\longrightarrow C$ of a polyhedron $C$ from this partition defines a 3-cycle in $H_{3}(\overline{X},\mathbb Z)$ called \emph{a polyhedral cycle}.

However, not every polyhedral cycle lifts to a cycle on the resolution of singularities $X$ of $\overline{X}$. Arrangements of eight planes that define  Calabi-Yau threefolds have eight possible types of singularities. A 3-cycle in $H_{3}(\overline{X},\mathbb Z)$ lifts to a 3-cycle in $H_{3}(X,\mathbb Z)$ if and only if it satisfies a symmetry condition at singular points of type $p_{4}^{0}$. A singularity of type $p_{4}^{0}$ is a point of intersection of four planes $P_{i}$, which are generic elsewhere 
(i.e. this point does not lie on a triple line).

Let $F(x,y,z,w)$ be the homogeneous equation of the octic arrangement $D$. Numerical integration of 
\[\iiint_{C}\frac{dxdydz}{\sqrt{F(x,y,z,1)}},\]
over all polyhedra $C$ gives period integrals of $\overline{X}$. Computed integrals generate a subgroup $\Lambda^{p}_{\overline X}$ of the group of period integrals of the singular double cover $\overline X$ of $\mathbb P^{3}$: 
\[\Lambda _{\overline{X}}:=\left\{\int_{\gamma}\omega: \gamma\in H_{3}(\overline{X},\mathbb 
Z)\right\}\]
In the following table we give real and complex generators of this subgroup. Note that for arrangements 19, 240 and 245 the group $\Lambda^{p}_{\overline{X}}$  is of rank $4$. This phenomenon is possible because the computations in \cite{Cynk-van Straten} are carried out on the singular double octic, which contains in these cases several $p_{4}^{0}$. Each $p_{4}^{0}$ point yields a condition on a cycle in $H_{3}(\overline X,\mathbb Z)$ to lift to a cycle in $H_{3}(X,\mathbb Z)$. If some $p_{4}^{0}$ points define independent conditions, the rank of $\Lambda^{p}_{\overline{X}}$ can be larger than 3.

	\begin{center}
	\begin{tabular}{r|rr|rr}
		Arr. &\multicolumn{2}{l|}{Real integrals}&
		\multicolumn{2}{l}{Imaginary integrals}\\\hline
		3&14.303841078& 18.695683053&
		41.413458745i\\
		19&12.3280533145& 19.3301891966&
		12.3280533145i& 19.3301891966i\\
		32&11.13352966& 16.85672240&
		17.34237466i\\
		69&11.13352966& 16.85672240&
		17.34237465i\\
		93&8.42836120319& 11.1335296603&
		17.3423746625i\\
		239&13.1823084825& 17.6714531944&
		11.7425210928i\\
		240&3.99263311132& 6.94406875218&
		4.80390756451i& 6.9176905115i\\
		245&3.99263311132& 6.94406875217&
		5.38024923409i& 7.49403218155i
	\end{tabular}
	\captionof{table}{Generators of $\textnormal{Re}\left(\Lambda^{p}_{\overline{X}}\right)$ and $\textnormal{Im}\left(\Lambda^{p}_{\overline{X}}\right)$ for rigid double octics with $\Lambda_{X}\subsetneq\Lambda^{p}_{\overline{X}}$}\label{tab:integrals}
\end{center}

\section{One-parameter families}\label{sec:one-p}

The main idea behind the method of computing period integrals introduced in \cite{Chmiel} is to consider a rigid Calabi-Yau manifold $X$ as a resolution of a degenerate element $X_{t_0}$ 
of a one-parameter family of smooth Calabi-Yau threefolds $\mathcal X = (X_t)_{t\in \mathcal B}$, where $\mathcal B=\mathbb P^{1}(\mathbb C)\setminus \Sigma$ is an open subset of $\mathbb P^{1}(\mathbb C)$.

As an explicit example let us begin with the case of double octics, which is our main interest in this paper. In this case the family $\mathcal X$ of Calabi-Yau threefolds is given by the corresponding family $D_{t}=\{(x,y,z,w)\in\mathbb P^{3}: F_{t}(x,y,z,w)=0\}$ of arrangements of eight planes inside $\mathbb{P}^3$. Then the Calabi-Yau threefold $X_{t}$ is a desingularization of the double cover $\overline{X_{t}}$ of $\mathbb P^{3}$ branched along $D_{t}$. A degeneration $X_{t_{0}}$ is an effect of an additional collision of planes in $D_{t_{0}}$ that does not happen in a generic fiber $D_{t}$. The sequence of blow-ups that resolves generic double cover $\overline{X_{t}}$ gives only 
a partial resolution $X_{t_{0}}$ of the double cover $\overline{X_{t_{0}}}$ in the degenerate fiber.

To be even more explicit, assume that four planes of the octic arrangement form a tetrahedron that shrinks to a point as $t\rightarrow t_{0}=0$. In appropriate coordinates we get \[F_{t}(x,y,z,w)=xyz(x+y+z-tw)G_{t}(x,y,z,w),\] with $G_{0}(0,0,0,1)\not =0$. Then the singular fiber is nodal and the shrinking tetrahedron \[\{(x,y,z)\in\mathbb R^{3}: x\ge0,\ y\ge0,\ z\ge0,\ x+y+z\le t\}\] defines the \emph{vanishing cycle} $\delta\in H_{3}(X_{t},\mathbb Z)$.

The difference between the resolution $X$ of the double cover $\overline{X_{t_{0}}}$ and its partial resolution $X_{t_{0}}$ is that in the former case we first blow-up the $p_{4}^{0}$ points and then the double lines, while in the latter we blow-up only the double lines (see \cite{Cynk-van Straten} for details). As a consequence, the variety 
$X_{t_{0}}$ is nodal with two nodes corresponding to each $p_{4}^{0}$ point. These nodes admit a small (crepant) resolution and exceptional lines are equivalent in $H^{4}(X,\mathbb C)$. In this situation we have $b_3(X_t)=4$, $b_3(X_{t_0})=3$ and $b_3(X)=2$. The homology group $H_3(X_t,\mathbb{Z})$ is spanned by $H_3(X_{t_0},\mathbb{Z})$ and the class of the vanishing cycle $\delta$.

Now let us for a moment return to a general situation of a family $\mathcal X = (X_t)_{t\in \mathcal B}$. If we fix a holomorphic family of $3$-forms $\omega_t\in H^{3,0}(X_{t})$ and a cycle $\delta\in H_3(X_{t_0},\mathbb{Z})$, in a punctured neighbourhood of ${t_0}\in\Sigma$ we can consider a (locally) holomorphic function $y(t):=\int_\delta\omega_t$, called the \textit{period function} of this family. For any loop $\gamma\in\pi_1(\mathcal{B},b)$, where $\mathcal{X}\rightarrow\mathcal{B}$ is the total space of the family and $b\in\mathcal{B}$ is some base point close to ${t_0}$, we can continue $y$ analytically along $\gamma$ and obtain a new function which we denote $M_\gamma(y)$. It turns out that the periods of $X$ 
can be recovered from the values $M_\gamma(y)({t_0})$.

A period function $y$ satisfies a fourth order differential equation called the \textit{Picard-Fuchs operator} of the family $X_t$ (see, e.g. \cite{Gross-Huybrechts-Joyce}). For any regular point $b$ the space of solutions of $\mathcal{P}=0$ near $b$ is four-dimensional and the fundamental group $\pi_1(\mathcal{B},b)$ acts on it by analytical continuation. After a choice of basis, this action defines the \textit{monodromy group} $\text{Mon}(\mathcal{P})\subset GL(4,\mathbb{C})$ of the operator $\mathcal{P}$. Every boundary point $s\in\mathcal{B}$ of the family $\mathcal X$ has associated local monodromy operator $M_s\in\text{Mon}(\mathcal{P})$, given by a small loop encircling $s$ counter-clockwise. If the local monodromy $M_s$ has the Jordan form
\begin{equation*}
	\begin{pmatrix}
		1&0&0&0\\0&1&1&0\\0&0&1&0\\0&0&0&1\\
	\end{pmatrix},
\end{equation*}
the singular point $s$ is called a \textit{conifold point}. Local exponents of a Picard-Fuchs operator at a conifold point equal $(0,1,1,2)$. 
By direct inspection we verify that the local exponents $(0,1,1,2)$ of the Picard-Fuchs operators of one-parameter families of double octics correspond exactly to a degeneration of the described type (introducing a new $p_{4}^{0}$ from a shrinking tetrahedron).

\begin{remark}
	Picard-Fuchs operators of a one-parameter family of double-octic Calabi-Yau threefolds have also singular points with local exponents $(0,\tfrac{1}{2},\tfrac{1}{2},1)$ and $(0,\tfrac{1}{4},\tfrac{1}{4},\tfrac{1}{2})$. The geometry of the degenerate fiber in this case can be more complicated. After a quadratic or quartic base-change, totally ramified at such singularity, we can get an equation with local exponents $(0,1,1,2)$. 
	As a consequence, the monodromy behaviour in this case is much better understood (see \cite{Chmiel2}). On the geometric side, we expect that a singular point of type $\frac1nC$, i.e. with local exponents $(0,\tfrac{1}{n},\tfrac{1}{n},\tfrac{2}{n})$, corresponds to a nodal Calabi-Yau threefold after a semi-stable base-change. Here for simplicity we focus on singularities of type $C$.
\end{remark}

The main advantage of the approach via Picard-Fuchs operators is that we do not depend on a geometric description of the considered rigid Calabi-Yau threefold. Assume that $\mathcal{P}$ is a Fuchsian differential operator of order $4$ such that $t_0$ is a conifold singularity. The image $\operatorname{Im}(M_{t_0}-\operatorname{Id})$ is one-dimensional. A generator of this subspace is called the \textit{conifold period} and we denote it by $f_c$; in the case of Picard-Fuchs operators it corresponds to the integral over the vanishing cycle.

Now we define
\begin{equation}{\label{eq:def}}
	\mathcal{L}_{\mathcal{P},t_0}:=\Big\langle\big\{M(f_c)(t_0): M\in \text{Mon}(\mathcal{P})\big\}\Big\rangle.
\end{equation}
The group $\mathcal{L}_{\mathcal{P},t_0}$ is only defined up to scaling, since in the definition we have to choose a specific conifold period $f_c$. 
In order to avoid this ambiguity we therefore normalize it by choosing the conifold period satisfying the condition $f_c(t)=(t-t_0)+O\left((t-t_0)^2\right)$. The group $\mathcal{L}_{\mathcal{P},t_0}$ contains limits of periods integrals of fibers $X_{t}$ as $t\rightarrow t_{0}$. They are periods of the singular variety $X_{t_{0}}$ but not necessarily periods of $X$. Consequently we get only the inclusion $\mathcal{L}_{\mathcal{P},t_0}\subset \Lambda _{X_{t_0}}$.

Assume that $t_{0}\in\mathbb Q$ and that the Picard-Fuchs operator $\mathcal P$ of the family $\mathcal X$ has rational coefficients: \[\mathcal P=P_{4}(t)D^{4}+P_{3}(t)D^{3}+\dots+P_{0}(t),\quad P_{i}\in\mathbb Q[T]\]
In this situation the Frobenius basis of solutions of $P$ at $t_{0}$ has the form
$$f_{1}(t-t_{0}),\ f_{2}(t-t_{0}),\ f_{3}(t-t_{0})+f_{2}(t-t_{0})\cdot\log(t-t_{0}),\ f_{4}(t-t_{0})$$
with $f_{1}\in\mathbb Q[[T]], f_{2},f_{3}\in T\mathbb Q[[T]]$ and $f_{4}\in T^{2}\mathbb Q[[T]]$. Consequently the space 
$\mathcal L_{P,t_{0}}$ is invariant under complex conjugation and 
\[2\cdot\Big(\textnormal{Re}\left(\mathcal L_{\mathcal P,t_{0}}\right)\oplus \textnormal{Im}\left(\mathcal L_{\mathcal P,t_{0}}\right)i\Big)\subset\mathcal L_{\mathcal P,t_{0}}\subset\textnormal{Re}\left(\mathcal L_{\mathcal P,t_{0}}\right)\oplus \textnormal{Im}\left(\mathcal L_{\mathcal P,t_{0}}\right)i,\] 
where \[\textnormal{Re}\left(\mathcal L_{\mathcal P,t_{0}}\right):=\{\textnormal{Re}(v): v\in\mathcal L_{\mathcal P,t_{0}}\},\;\textnormal{Im}\left(\mathcal L_{\mathcal P,t_{0}}\right):=\{\textnormal{Im}(v): v\in\mathcal L_{\mathcal P,t_{0}}\}\]
denote the real and complex parts of $\mathcal{L}_{P,t_{0}}$.

In general, it is difficult to compute the monodromy group of a Picard-Fuchs operator. 
Assume that the family $\mathcal X$ has a point of Maximal Unipotent Monodromy (MUM) at $t=0$. A choice of a path connecting points 0 and $t_{0}$, while avoiding other singularities of the Picard-Fuchs operator $\mathcal P$, gives a subgroup 
\begin{equation}
	\mathcal L^{0}_{\mathcal P,t_{0}}:=\langle \{\textnormal{Re}(M_{0}^{n}): n\in\mathbb Z\}\rangle + \langle \{\textnormal{Im}(M_{0}^{n})i: n\in\mathbb Z\}\rangle \subset\mathcal L_{\mathcal P,t_{0}}
\end{equation}
defined by a local monodromy $M_{0}$ of $\mathcal P$ around the MUM point $t=0$. 

Coming back to the case of double octics, in \cite{Chmiel} we observed (numerically) that for all one-parameter families of double octics the group of real periods $\textnormal{Re}\left(\mathcal L^{0}_{P,t_{0}}\right)$ has rank one, while in nine cases the subgroup of imaginary periods $\textnormal{Im}\left(\mathcal L^{0}_{P,t_{0}}\right)$ has rank 2. Thus in this case the inclusion $\Lambda_X\subset\mathcal{L}^0_{\mathcal{P},t_0}$ is strict; note that due to results from \cite{Chmiel2} this cannot happen for a singularity of type $\tfrac{1}{2}C$.
Moreover, the generators of $\mathcal{L}^{0}_{\mathcal{P},t_0}$  (Table \ref{tab:monodromy}) and the generators of $\Lambda^{p}_{\overline{X}}$ (Table \ref{tab:integrals}) can be expressed in terms of each other. 

 \def\arraystretch{1.5}
\begin{longtable}{|c|c|c|c|c|}
	\hline
	Operator&Conifold point&Rigid Arr.&Form&Generators of $\textnormal{Im} \mathcal L^{0}_{P,t_{0}}$\\
	\hline
	\multirow{2}{*}5&\multirow{2}{*}0&\multirow{2}{*}3&\multirow{2}{*}{32/2}&3.78853747194184773010686231258i\\
	&&&&61.0738884585292464400038239965i\\\hline
	\multirow{2}{*}5&\multirow{2}{*}2&\multirow{2}{*}3&\multirow{2}{*}{32/2}&0.94713436798546193252671571987i\\
	&&&&3.90285969880676841968001329994i\\\hline
	\multirow{2}{*}{20}&\multirow{2}{*}{-2}&\multirow{2}{*}{19}&\multirow{2}{*}{32/1}&0.979278824715794481666000593885i\\
	&&&&4.45674355709313141111341743112i\\\hline
	\multirow{2}{*}{95}&\multirow{2}{*}{$-\tfrac12$}&\multirow{2}{*}{93}&\multirow{2}{*}{8/1}&
	2.99683078705084653614316487029i\\
	&&&&34.0238543159967756814545903982i \\\hline
	\multirow{2}{*}{244}&\multirow{2}{*}{$\tfrac12$}&\multirow{2}{*}{240}&\multirow{2}{*}{6/1}&
	2.58823590805561845768157001028i\\&&&&32.0498374325403392826453731746i\\\hline
	\multirow{2}{*}{244}&\multirow{2}{*}{2}&\multirow{2}{*}{240}&\multirow{2}{*}{6/1}&10.3529436322224738307260280020i\\*
	&&&&128.199349730161357130578977414i\\* \hline
	\multirow{2}{*}{253}&\multirow{2}{*}{-2}&\multirow{2}{*}{245}&\multirow{2}{*}{6/1}&
	6.26847094349121003359079492495i\\&&&&7.96011334055139325749281017005i\\\hline
	\multirow{2}{*}{274}&\multirow{2}{*}{$-\tfrac12$}&\multirow{2}{*}{245}&\multirow{2}{*}{6/1}&
	0.839792675513409448977564062085i\\&&&&8.12249522907253404678014309610i\\\hline
	\multirow{2}{*}{274}&\multirow{2}{*}{-2}&\multirow{2}{*}{245}&\multirow{2}{*}{6/1}&
	1.49296475661811062877124575360i\\&&&&14.4399915143798181025349619989i\\\hline
	\caption{Generators of $\mathcal{L}^{0}_{\mathcal{P},t_0}$ for singular points $t_{0}$ with $\operatorname{rank}\left(\mathcal{L}^{0}_{\mathcal{P},t_0}\right)=3$}
	\label{tab:monodromy}
\end{longtable}

Obviously the group $\mathcal{L}_{\mathcal{P},t_0}$ depends not only on the smooth double octic $X$ but also on the choice of a one-parameter smoothing. In fact birational models of rigid double octic can be realized as specializations of several one-parameter families. Up to commensurability, we always have inclusions $\Lambda_X\subset\mathcal{L}_{\mathcal{P},t_0}\subset\Lambda^{p}_{X_{t_0}}\subset\Lambda^{p}_{\overline{X}}$. As we already mentioned, there are cases when $\Lambda_X=\mathcal{L}^{0}_{\mathcal{P},t_0}$ but there are also examples in which $\text{rank}(\mathcal{L}_{\mathcal{P},t_0})=3$.
Additional period integrals in $\Lambda^{p}_{\overline{X}}$ are related to singular points of type $p_{4}^{0}$. The classification in \cite{Cynk-Kocel} shows that $\Lambda^{p}_{\overline{X}}$ is the sum of $\Lambda^{p}_{X_{t_0}}$ taken over all one parameter smoothings $X_{t}$ of $X$.

Results of \cite{Chmiel} have two important consequences: using 
Maple implementation of algorithms for solving differential equations and numerical approximations to construct an analytic continuation along any polyline path, we can compute the elements of $\mathcal{L}_{\mathcal{P},t_0}$ with precision of hundreds of digits. This is in striking contrast with computations in \cite{Cynk-van Straten}, where only precision of $10$ digits could be obtained. Moreover, the definition (\ref{eq:def}) is given purely in terms of the differential equation and thus allows us to assign an analogue of $\Lambda_{\overline{X}}$ to any smooth Calabi-Yau threefold $X$ which is birational to a degeneration of a family of Calabi-Yau threefolds $X_t$ with $h^{2,1}(X_t)=1$.
The question of understanding period integrals of singular models  $\overline{X}$ of a rigid Calabi-Yau threefold $X$ is therefore replaced with a more general problem of describing the elements of $\mathcal{L}_{\mathcal{P},t_0}$, where $\mathcal{P}$ is the Picard-Fuchs operator of a one-parameter smoothing of a singular model $X_{t_0}$ of $X$. We want to accomplish it in terms of the modular form associated to $X$ by the modularity theorem.

\section{Twists by a Dirichlet character}\label{sec:twists}

A rigid Calabi-Yau threefold can have models which are isomorphic over a number field but not isomorphic over $\mathbb Q$. In this situation the associated modular form is not uniquely determined by its model over complex numbers. 
Since a double octic Calabi-Yau threefold is hyperelliptic, it admits a quadratic twist by any square-free integer. Quadratic twists exist for a large class of Calabi-Yau threefolds including double octics and Schoen's fiber products. However, 
existence of quadratic twists for an arbitrary Calabi-Yau threefold is an 
open question.  In \cite{GKY} Gou\^eva, Kiming and Yui proposed an abstract definition of a quadratic twist.

To be more explicit, if a rigid double octic $X$ is given as a resolution of the hypersurface
\[\{u^{2}=f(x)\}\subset\mathbb P(1^{4},4),\]
then there exists a quadratic twist $X_{d}$ by a square-free integer $d$  given by a resolution of
\[\{u^{2}=d\cdot f(x)\}\subset\mathbb P(1^{4},4).\]
Threefolds $X$ and $X_{d}$ are obviously isomorphic over $\mathbb Q[\sqrt{d}]$, but they are not isomorphic over $\mathbb Q$, unless 
the corresponding modular form is of CM-type.

If $X$ is a rigid Calabi-Yau manifold defined over $\mathbb Q$ with attached modular form
\[f(z)=\sum _{n=1}^{\infty} a_{n}q^{n}\in\Gamma_{0}(N), \quad q=\exp(2\pi iz),\]
then the modular form associated to a quadratic twist $X_{d}$ by $d$ is
\[f_{\chi_{d}}(z)=\sum _{n=1}^{\infty} a_{n}\chi_{d}(n)q^{n}\in\Gamma_{0}(N), \quad q=\exp(2\pi iz).\]
Thus it is the quadratic twist of $f$ by the Dirichlet character $\chi_{d}$ (\cite[Thm.~1]{GKY}).

Let $g(\chi_{d})$ be the Gauss sum of a Dirichlet character $\chi_{d}$ modulo $d$:
\[g(\chi_{d}):=\sum_{a=1}^{d}\chi_{d}(a)\exp\left(\frac{2\pi ia}{d}\right)\]
If the character $\chi_{d}$ is primitive then $|g(\chi_{d})|=\sqrt{d}$. For a Dirichlet character $\chi$ denote by $K_\chi$ the field of definition of $\chi$.
We have the following formulas for the special values of a twist $f_{\chi}$ of the modular form $f$ by a Dirichlet character:
\begin{theorem}[\mbox{\cite[Thm. 1]{Shimura}}]
	Let $f\in S_{k}(\Gamma_{0}(N))$ be a cusp form of weight $k$ for the group $\Gamma_{0}(N)$ with rational coefficients. 	
	There exist complex numbers $u^{+}$ and $u^{-}$ such that 
	for any  Dirichlet character $\chi$ and any positive integer $m<k$ we have
	\[(2\pi i)^{-m}g(\chi)^{-1}L(f_{\chi},m)\in\begin{cases}
	u^{+}K_{\chi}, \; &\text{ if }\chi(-1)=(-1)^{m}\\
	u^{-}K_{\chi},  &\text{ if }\chi(-1)=-(-1)^{m}
	\end{cases}\]
\end{theorem}
In the special case of a quadratic twist $X_{d}$ of a Calabi-Yau threefold and the attached modular forms $f$ and $f_{d}$ 
we get
\begin{eqnarray*}
	&&\text{If } d>0, \text{ then } L(1,f_{d})\in \sqrt{d}\cdot L(1,f)\cdot\mathbb Q, \quad L(2,f_{d})\in \sqrt{d}\cdot L(2,f)\cdot\mathbb Q\\
	&&\text{If } d<0, \text{ then } L(1,f_{d})\in \frac{\sqrt{d}}{2\pi i}\cdot L(2,f)\cdot\mathbb Q, \quad L(2,f_{d})\in \sqrt{d}\cdot 2\pi i\cdot L(1,f)\cdot\mathbb Q
\end{eqnarray*}
This formulas agree with the behaviour of period integrals of a double octic Calabi-Yau threefold under a quadratic twists, since the period integrals of $X_{d}$ equal period integrals of $X$ divided by $\sqrt{d}$. In particular, when $d$ is negative, real periods of $X$ correspond to complex periods 
of $X_{d}$ and vice versa.

As we mentioned in section \ref{sec:1}, period integrals of a rigid Calabi-Yau threefold $X$ defined over $\mathbb{Q}$ are expected to be proportional 
to special values of the $L$-function of the corresponding modular form and thus we want to interpret the elements of $\Lambda_{\overline{X}}$ in a similar way. 

However, we have to take into account that Calabi-Yau threefolds isomorphic over $\mathbb C$ need not have equal associated modular forms, as they can fail to be isomorphic over $\mathbb Q$. This phenomenon can occur also for one-parameter families of Calabi-Yau threefolds.
Consequently, given a differential operator $\mathcal{P}$, it may happen that we find two families $\mathcal{X}$ and $\mathcal{Y}$ having $\mathcal{P}$ as the Picard-Fuchs operator, yet such that the smooth, rigid models of singular fibers at a conifold point have associated modular forms equal only up to a twist. It is then not \textit{a priori} clear which modular form should the elements of $\Lambda_{\overline{X}}$ be compared with.

In a similar manner the distinction between different twists of the same manifold is not visible from the differential operator. The Picard-Fuchs operator $\mathcal{P}$ of a family $X_t^a$ given by the equation $x_0^2=aF_t(x_1,x_2,x_3,x_4)$ is independent of the choice of $a\in\mathbb{C}^*$. Indeed, the preferred choices of the period function $$\omega_t^a=\int_\gamma\tfrac{\sum_{i=1}^{4}(-1)^ix_idx_1\wedge\cdots\wedge\widehat{dx_i}\wedge\cdots dx_4}{x_0}=\int_\gamma\tfrac{\sum_{i=1}^{4}(-1)^ix_idx_1\wedge\cdots\wedge\widehat{dx_i}\wedge\cdots dx_4}{\sqrt{aF_t(x_1,x_2,x_3,x_4)}},$$ differ only by a scalar and thus satisfy the same differential equation. Therefore if we are only given the Picard-Fuchs operator and not the family itself, it is not possible to determine from which of the families $X_t^a$ it comes.

What we may do, however, is to normalize the conifold period $f_c$ so that $f_c(t)=(t-t_0)+O\left((t-t_0)^2\right)$. This normalization is usually used in the descriptions of the Frobenius method, it was also in place for computations in \cite{Chmiel} and in our definition (\ref{eq:def}). Note that this choice happens on the level of the differential equation and not on the level of the family. Then inside $\mathcal{L}_{\mathcal{P},t_0}\otimes\mathbb{Q}$ we can identify the lattice $\lambda\Lambda_f$ for some $\lambda\in\mathbb{C}^*$, where $f$ is a modular form associated to \textit{some} rigid birational model of $X_{t_0}$. It is then natural to assume that a modular form $f_{t_0}$ associated with the point $t_0$ is the one for which $\Lambda_{f_{t_0}}$ and $\lambda\Lambda_f$ are commensurable. One may consult the Table 4 from \cite{Chmiel} to see examples for which $\lambda\neq1$ and thus the modular form of 
the minimal level associated to the rigid model of $X_{t_0}$ by Meyer is not the one for the singular model in the sense just described.

\section{Invariants of the modular form}\label{sec:M(f,k)}

We now consider the associated modular forms in order to check whether we 
can describe elements of $\mathcal{L}_{\mathcal{P},t_0}$ in a way similar 
as we describe periods of a rigid Calabi-Yau threefold in terms of $L(f,1)$ and $\tfrac{L(f,2)}{2\pi i}$. To this end we have to embed $\Lambda_f$ 
into some intrinsically defined group of greater rank, and therefore it is natural to try and write the special values generating $\Lambda_f$ as sums of some invariants of $f$.

Let $W_N$ be the Fricke involution on the space of modular cusp forms $\mathcal{S}_4(\Gamma_0(N))$  of weight $4$ and level $N$. It is a linear operator defined by $W_N(f)(z):=N^{-2}z^{-4}f(\tfrac{-1}{Nz})$. One easily checks that $W_N$ is an idempotent, i.e. $W_N^2=\operatorname{Id}$. If $f$ is a modular form associated with a rigid Calabi-Yau threefold, it is also an eigenvector of $W_N$ and hence $W_N(f)=\varepsilon\cdot f$, where $\varepsilon=\pm 1$ is the \textit{Fricke sign} of $f$.

The completed $L$-function  $\Lambda(s):=(\tfrac{\sqrt{N}}{2\pi})^s\Gamma(s)L(f,s)$ of a Hecke eigenform $f$  satisfies the functional equation $\Lambda(s)=\varepsilon\Lambda(4-s)$. The completed $L$-function can be also defined in terms of the Mellin transform of $f$ as $\Lambda(s):=\sqrt{N}^s\int\limits_{0}^{\infty}f(iz)z^{s-1}dz$.
The decomposition of a cycle on a smooth Calabi-Yau threefold into a sum of cycles in its nodal model is given by splitting at the node. Motivated by this description we define a 'partial' $L$-function of a modular form:
\begin{definition}
Let $f\in\mathcal{S}_k(\Gamma_0(N))$ be a Hecke eigenform. Then we define
\begin{equation}\label{definiton}
M(f,s):=\frac{(2\pi)^s}{\Gamma(s)}\int\limits_{\sqrt{N}^{-1}}^{\infty}f(iz)z^{s-1}dz.
\end{equation}
\end{definition}	
\noindent 
From the standard proof of the functional equation for $\Lambda(s)$ we can deduce the following property of the function $M(f,s)$, fundamental for our goal:

\begin{theorem}\label{th:eq}
	If $f\in\mathcal{S}_k(\Gamma_0(N))$ is an eigenform of the Fricke involution, then $$L(f,s)=M(f,s)+\varepsilon(\tfrac{2\pi}{\sqrt{N}})^{2s-k}\tfrac{\Gamma(k-s)}{\Gamma(s)}M(f,k-s)$$
\end{theorem}

\begin{proof}
We have the following chain of equalities:
\begin{eqnarray*}
	&&\Lambda(s)=N^{\tfrac{s}{2}}\int_{0}^{\infty}f(iz)z^{s-1}dz=N^{\tfrac{s}{2}}\int_{0}^{\sqrt{N}^{-1}}f(iz)z^{s-1}dz+N^{\tfrac{s}{2}}\int_{\sqrt{N}^{-1}}^{\infty}f(iz)z^{s-1}dz\\
	&&=\varepsilon N^{\tfrac{k-s}{2}}\int_{\sqrt{N}^{-1}}^{\infty}f(iz)z^{k-s-1}dz+N^{\tfrac{s}{2}}\int_{\sqrt{N}^{-1}}^{\infty}f(iz)z^{s-1}dz\\
\end{eqnarray*}
Now, the definition of $M(f,s)$ gives the assertion:
\begin{eqnarray*}
&&L(f,s)=(\tfrac{2\pi}{\sqrt{N}})^s\tfrac{1}{\Gamma(s)}\Lambda(s)\\
&&=\tfrac{(2\pi)^s}{\Gamma(s)}\int_{\sqrt{N}^{-1}}^{\infty}f(iz)z^{s-1}dz+\varepsilon(\tfrac{2\pi}{\sqrt{N}})^{2s-k}\tfrac{\Gamma(4-s)}{\Gamma(s)}\cdot\left(\frac{(2\pi)^{k-s}}{\Gamma(k-s)} \int_{\sqrt{N}^{-1}}^{\infty}f(iz)z^{k-s-1}dz\right)\\
&&=M(f,s)+\varepsilon(\tfrac{2\pi}{\sqrt{N}})^{2s-k}\tfrac{\Gamma(k-s)}{\Gamma(s)}M(f,k-s)
\end{eqnarray*}
\end{proof}

For our purposes crucial is the application of Theorem \ref{th:eq} to the 
special values of the $L$-function at the critical points $1$ and $2$. For $L(f,1)$ we obtain the decomposition
$$L(f,1)=M(f,1)+\varepsilon\frac{N}{2\pi^2}M(f,3),$$
which suggests that $M(f,1)$ and $\tfrac{M(f,3)}{\pi^2}$ might be the additional elements in $\Lambda_{\overline{X}}$ needed to decompose integrals of $X$ as previously described. On the other hand, for the special value $L(f,2)$ the situation is different. In the geometric context, the imaginary period is computed as integral over 
a cycle in $H_{3}(X)$ that is not decomposed into a sum of two cycles in  
$H_{3}(\overline{X})$. 
On the modular side, Theorem \ref{th:eq} in this case yields
$$L(f,2)=(1+\varepsilon)M(f,2)$$
When $\varepsilon=-1$, this obviously implies $L(f,2)=0$. 
Thus in this case $\Lambda_f$ is not a lattice but a group of rank $1$ and the proportionality of the special $L$-value $L(f,2)=0$ with the period integral is trivial, hence meaningless. If $\varepsilon=1$, then $M(f,2)=\tfrac{L(f,2)}{2}$ 
which means that the Fricke involution divides the 3-cycle computing imaginary period into two subsets of equal $\omega$-volume.
Similarly, adding $\tfrac{M(f,2)}{2\pi i}$ to $\Lambda_f$ results in a commensurable lattice.

Thus let us define $$\Lambda_f^c:=\langle M(f,1),\tfrac{L(f,2)}{2\pi i},\tfrac{M(f,3)}{2\pi^2}\rangle$$
As we have seen, $\Lambda_f\subset\Lambda_f^c$ and we hope that $\Lambda_f^c$ can play the role of a 'modular' analogue of $\mathcal{L}_{\mathcal{P},t_0}$. The last section will present numerical evidence supporting this hypothesis, as well as comment on aforementioned problems with its direct application to 
the case of the form $6/1$.

\section{Modular interpretation of additional integrals}\label{sec:interpretation}

In this section we shall propose a conjectural relation between the integrals $M(f,1)$ and $\tfrac{M(f,3)}{2\pi^2}$ and the period integrals of the degenerate element of a one-dimensional family. 
As in Section \ref{sec:one-p}, we shall consider a one-parameter family $\mathcal X=(X_{t})_{t\in \Delta}$ of projective  varieties such that $t\not\in\{0,t_{0}\}$ the variety $X_{t}$ is a (smooth) Calabi-Yau threefold with $h^{1,2}(X_{t})=1$. We also assume that $t=0$ is a point of Maximal Unipotent Monodromy, $t=t_{0}$ is a conifold point and that the degeneration $X_{t_{0}}$ admits a crepant resolution of singularities $X$ which is a rigid Calabi-Yau manifold defined over $\mathbb Q$.

Recall that there exist eleven rigid double octics defined over $\mathbb{Q}$. Among them there are eight examples for which integrals over polyhedral cycles on the singular double cover of $\mathbb{P}^3$ generate a group of rank greater than $2$; they correspond to arrangements $3,19,32,69,93,239,240$ and $245$. 
Twists of minimal level of modular forms associated with these examples are: 
\begin{itemize}\labelsep=1cm\leftskip=2cm
	\item [$6/1$:] Arr. 240, 245, 
	\item [$8/1$:] Arr. 32, 69, 93, 
	\item [$12/1$:] Arr. 239, 
	\item [$32/1$:] Arr. 19,
	\item [$32/2$:] Arr. 3.
\end{itemize}

\begin{table}[h]
	\[\begin{array}{r|cc}
		\text{Form}&M(f,1)\approx&M(f,3)\approx\\
		\hline		
6/1&
0.0705795645108305473225513900913349620339&
0.4969159973924347680403980021800044307499
\\
8/1&
0.1067464416589652341658482119578535772357&
0.6113090894855911122392465328302544499540
\\
12/1&
0.1669558873241197456891912776162834226044&
0.7363105497640607198175099780042098699693
\\
32/1&
0.3462273488144653581186531731448149867719&
0.9131253716041011279777950209397188200763
\\
32/2&
0.4179438555831683670296074873254014692256&
0.9969238636883323441309149178659389533592
	\end{array}
	\]\caption{}
	\label{tab:M}
\end{table}

For all rigid double octics, except those with the modular form 6/1, we have the following relations between additional integrals (listed in Table \ref{tab:integrals}) and the invariants $M(f,1)$ and $M(f,3)$ of the corresponding modular form (listed in Table  \ref{tab:M}):

\begin{prop}\label{observation}
Up to precision of computations in \cite{Cynk-van Straten} we have the following relations between additional period integrals for the singular double cover $\overline{X}$ and the invariants $M(f,1)$ and $M(f,3)$ of the modular form associated to the corresponding double octic Calabi-Yau threefold $X$:
\begin{eqnarray*}
	\text{Arr. 3}&&\begin{cases}
		14.30384107=-12\pi^2M(f_{32/2},1)+64M(f_{32/2},3)\\
		18.69568305=20\pi^2M(f_{32/2},1)-64M(f_{32/2},3)\\
	\end{cases}\\
	\text{Arr. 19}&&\begin{cases}
		12.32805331=-6\sqrt{2}\pi^2M(f_{32/1},1) + 32\sqrt{2}M(f_{32/1},3)\\
		19.33018919=4\sqrt{2}\pi^2M(f_{32/1},1)
	\end{cases}\\
	\text{Arr. 32, 69}&&\begin{cases}
		11.13352966=-8\pi^2M(f_{8/1},1) + 32M(f_{8/1},3)\\
		16.85672240=16\pi^2M(f_{8/1},1)\\
	\end{cases}\\
	\text{Arr. 93}&&\begin{cases}
		8.428361203=8\pi^2M(f_{8/1},1)\\
		11.13352966=-8\pi^2M(f_{8/1},1)+32M(f_{8/1},3)
	\end{cases}\\
	\text{Arr. 239}&&\begin{cases}
		13.18230848=8\pi^2M(f_{12/1},1)\\
		17.67145319=24M(f_{12/1},3)
	\end{cases}
\end{eqnarray*}
Consequently in all those cases $\Lambda_f^c=\Lambda_{\overline{X}}$, up to commensurability.
\end{prop}

The precision of period integrals in Prop.~\ref{observation} is limited due to the method used in \cite{Cynk-van Straten}. However, by \cite{Chmiel} these period integrals agree (up to their exactness) with period integrals $\mathcal{L}^0_{\mathcal{P},t_0}$ computed via the analytic continuation of a conifold period. Consequently the groups $\mathcal L^{0}_{\mathcal P,t_{0}}$ and $\Lambda _{f}^{c}$ are commensurable as well.

For instance the relation
	\[14.303841078\approx\frac{\pi^2}{16}\cdot(61.0738884585292464400038239965-10\cdot3.78853747194184773010686231258)\]
suggests equalities
\[61.0738884585292464400038239965 = 208M(f,1)-256\frac1{\pi^2}M(f,3)\]
\[3.78853747194184773010686231258 = 40M(f,1)-128\frac1{\pi^2}M(f,3)\]
\[M(f,1)=\frac1{128}61.0738884585292464400038239965 -\frac1{64}3.78853747194184773010686231258	\]
\[M(f,3)=\frac{5\pi^2}{2048}61.0738884585292464400038239965 -\frac{13\pi^2}{1024}3.78853747194184773010686231258	\]
Since both elements of $\mathcal L^{0}_{\mathcal P,t_{0}}$ and $M(f,k)$, $k=1,2$, can be computed with very high precision, this equalities can be easily verified with accuracy $10^{-100}$ and higher, unlike those in Proposition \ref{observation}.

The case of Arr. No. 19 is exceptional, because we have commensurability of groups $\mathcal L^{0}_{\mathcal P,t_{0}}$ and  $\sqrt2\Lambda^{c}_{f_{32/1}}$. In this situation we expect that the double octic corresponds to the modular form $f_{64/1}$ which is the twist of $f_{32/1}$ by Dirichlet character $\chi_{8,5}$ or $\chi_{8,3}$. Modular form $f_{32/1}$ has complex multiplication by $\mathbb Q[\sqrt{-1}]$ hence it is invariant under the twist by the character $\chi_{4,3}$. As a consequence we cannot distinguish twists of $f_{32/1}$ by odd and even character. 
	\[\begin{array}{|c|c|}\hline
	M(f_{64/1},1)\approx0.366733368496185708303364416057\ &M(f_{64/1},3)\approx0.909804035050076966996940010381\\
	\hline		
\end{array}
\]
The Fricke sign for modular form $f_{64/1}$ equals $-1$. In particular $L(f_{64/1},2)=0$ and in this situation we do not predict that groups $\mathcal L^{0}_{\mathcal P,t_{0}}$ and  $\Lambda^{c}_{f}$ are commensurable. 
\begin{cor}
	For octic arrangements No. 3, 32, 69, 93, 239 and modular forms 32/2, 8/1, 8/1, 8/1, 12/1 respectively, groups
	 $\mathcal L^{0}_{\mathcal P,t_{0}}$ and $\Lambda _{f}^{c}$
	 are commensurable. For the octic arrangement No. 19 groups $\mathcal L^{0}_{\mathcal P,t_{0}}$ and  $\sqrt2\Lambda^{c}_{f_{32/1}}$
	 are commensurable. 
\end{cor}

In the above corollary commensurability means that the generators of one group can be expressed as integral linear combinations of generators of the second group with very high accuracy.

\subsection{Operator No. 8.62}
Consider the differential operator no. 8.62 in the online database \cite{CYDB}:
\begin{gather*}
\mathcal P=\theta^4+x\left(578\theta^4-572\theta^3-359\theta^2-73\theta-6\right)\\+3^{2} x^{2}\left(4673\theta^4+1892\theta^3+31601\theta^2+11514\theta+1728\right)\\-2^{3} 3^{4} x^{3}\left(9185\theta^4-134298\theta^3-35420\theta^2-22329\theta-5544\right)\\+2^{4} 3^{8} x^{4}\left(19051\theta^4+11846\theta^3+114678\theta^2+65939\theta+14290\right)\\-2^{6} 3^{12} x^{5}\left(7540\theta^4+8068\theta^3-6459\theta^2-7907\theta-2300\right)\\-2^{6} 3^{16} x^{6}\left(3919\theta^4+27744\theta^3+29957\theta^2+14208\theta+2556\right)\\+2^{9} 3^{20} 5 x^{7}\left(199\theta^4+590\theta^3+744\theta^2+449\theta+106\right)\\-2^{12} 3^{24} 5^{2} x^{8}\left((\theta+1)^4\right)	
\end{gather*}
It is the Picard-Fuchs operator of the one-parameter family of Calabi-Yau manifolds constructed as resolution of singularities of fiber products of semistable rational elliptic surfaces (see \cite{Schoen}) with singular fibers matched as in the following diagram:
\[\begin{matrix}
	I_5&I_3&I_2&I_1&I_1&-\\
	I_3&I_6&I_2&-&-&I_1
\end{matrix}\]
For the special value of the parameter $t_0=-\tfrac{1}{81}$ fibers $I_5$ and $I_1$ in the first surface collide producing a fiber of type $I_6$.
Consequently, the family contains degeneration at the conifold point given by a fiber product of the following Beauville surfaces 
\[\begin{matrix}
	I_6&I_3&I_2&I_1&-\\
	I_3&I_6&I_2&-&I_1
\end{matrix}\]
Since a generic fiber product in this family has 37 nodes while the special one has 40 nodes, the degenerate element $X_{t_{0}}$ of the family of smooth Calabi-Yau manifolds has 3 nodes. A small resolution of $X_{t_{0}}$ is the Calabi-Yau manifold $W_{2}$ constructed by Sch\"utt and its associated modular form $f$ is the form $21/2$ of weight $4$ and level $21$ (see \cite{Schutt}).

In this case the group \[\mathcal{L}^{0}_{\mathcal{P},-\frac{1}{81}}
=\langle 0.079041901426502594058424764412257593, 0.13670990041323305298936699557707682i\rangle
\] has rank 
$2$ and 
\[0.079041901426502594058424764412257593\approx \frac4{27}L(f_{21/1},1)\]
so we expect that the real generator should be expressible by the invariants of the modular form $f_{336/7}$ which is a twist of $f_{21/1}$ by $\chi_{4,3}$.
	\[\begin{array}{|c|c|}\hline
		L(f_{21/1},1)\approx0.53353283462889250989436715978273&
		L(f_{21/1},2)=0\\\hline
	M(f_{336/7},1)\approx0.415965257835022165771740339742\ &M(f_{336/7},3)\approx0.985536337116599774489352376439\\
	\hline		
\end{array}
\]
However, the Fricke sign of $f_{21/1}$ is $-1$ and consequently $L(f_{21/1},2)=0$ and again we cannot predict commensurability.

\subsection{Operator No. 8.67}
Consider the differential operator 
\begin{gather*}
	\mathcal P=5^{2} \theta^4+5 x\left(477\theta^4+978\theta^3+769\theta^2+280\theta+40\right)
	-2^{2}x^{2}\left(46\theta^4-2582\theta^3-5689\theta^2-4120\theta-1040\right)\\+2^{2} x^{3}\left(772\theta^4-4872\theta^3-11765\theta^2-7335\theta-1480\right)+2^{4} 3 x^{4}\left(140\theta^4+500\theta^3-672\theta^2-1313\theta-512\right)\\-2^{6} x^{5}\left(31\theta^4+154\theta^3-596\theta^2-729\theta-227\right)+2^{7} x^{6}\left(32\theta^4-264\theta^3-500\theta^2-303\theta-58\right)\\+2^{8} x^{7}\left(12\theta^4+72\theta^3+121\theta^2+85\theta+22\right)-2^{12} x^{8}\left((\theta+1)^4\right)
\end{gather*}
This operator has no. 8.67 in  \cite{CYDB}, it is the Picard-Fuchs operator of a family of resolutions of fiber products of the same elliptic surfaces as in the case of operator 8.62 
but with different matching of singular fibers
\[\begin{matrix}
	I_5&I_3&I_2&I_1&I_1&-\\
	I_6&I_2&I_3&-&-&I_1
\end{matrix}\]

For a special value of the parameter fibers $I_5$ and $I_1$ in the first surface collide producing a fiber of type $I_6$.
Degeneration at the conifold point $t_{0}=-1$ is a resolution of the fiber product  
\[\begin{matrix}
	I_6&I_3&I_2&I_1&-\\
	I_3&I_6&I_2&-&I_1
\end{matrix}\]
A small resolution of $X_{t_{0}}$ is a Calabi-Yau manifold $W_{1}$ from \cite{Schutt} and its associated modular form $f$ is the form $f_{17/1}$ of weight $4$ and level $17$ .

In this case $\text{rank}(\mathcal{L}_{\mathcal{P},-1})=3$. More precisely a finite index subgroup $\mathcal{L}_{\mathcal{P},-1}$ has one real and two imaginary generators
\[
\begin{array}{|c|c|}
	\hline
	\text{real generator}&\text{imaginary generators}\\
	\hline 42.906578481269266425208768540874910763200659&
	10.196661075170437602752538923890424049717592i\\
	\cline{2-2}
	&15.929100879959719028595857308255012478345243i\\
	\hline
\end{array}\]
In particular, the real generator equals
\[42.906578481269266425208768540874910763200659=-108L(f_{17/1},1) =\frac{36}\pi L(f_{272/4},2),\]
the form $f_{272/4}$ is the twist of $f_{17/1}$ by $\chi_{4,3}$.
Consequently, 
	\[\begin{array}{|c|c|}\hline
	L(f_{17/1},1)\approx-0.39728313408582654097415526426736028&
	L(f_{17/1},2)=0\\\hline
	M(f_{272/4},1)\approx0.12059537134699121108836815998156741 &M(f_{272/4},3)\approx0.60132423252344910630079471213129171\\
	\hline		
\end{array}
\]
In this case the Fricke sign of $f_{17/1}=-1$, hence $L(f_{17/1},2)=0$ and consequently we do not predict
commensurability.

\subsection{Arrangements no. 240 and 245}
Finally we shall go back to the most involved cases of rigid double octic Calabi-Yau threefolds with modular form $6/1$. In both cases the group $\Lambda_{\overline X}$ of periods of the singular double cover $\overline X$ has rank at least 4 with real and imaginary parts of rank at least 2.
The generator of $\textnormal{Re}(\mathcal L_{\mathcal P,t_{0}})$ for four conifold points $\tfrac{1}{2},2$ and $-\tfrac{1}{2},-2$ appearing in for Calabi-Yau operators $\mathcal P$ of arrangements no. 244 and 274 equals respectively (see \cite{Chmiel}):
\[20\sqrt{2}L(f_{6/1},1),\quad 80\sqrt{2}L(f_{6/1},1),\quad 10\sqrt{2}L(f_{6/1},1),\quad 40\sqrt{2}L(f_{6/1},1),\]
while for the conifold point $-2$ in arrangement no. 253 it is equal to
\[36\sqrt{2}\frac{L(f_{6/1},2)}{\pi}\]
Consequently, we consider the modular forms 
$f_{192/2}$ and $f_{192/7}$ which are twists of $f_{6/1}$ by Dirichlet characters
$\chi_{8,5}$ and $\chi_{8,3}$ respectively.
As \[L(f_{192/2},1)=-30\sqrt{2}L(f_{6/1},1) 
\quad\text{ and }\quad L(f_{192/7},1)=-36\sqrt{2}\frac{L(f_{6/1},2)}{\pi}
\]
generators of $\textnormal{Re}(\mathcal L_{\mathcal P,t_{0}})$
equal
\[-\tfrac23L(f_{192/2},1),\quad -\tfrac83L(f_{192/2},1),\quad -L(f_{192/7},1), \quad -\tfrac13L(f_{192/2},1),\quad -\tfrac{16}{27}L(f_{192/2},1).\]
The Fricke sign of both these forms is $-1$ and, as in the examples above, we do not predict that the generators of $\textnormal{Im}(\mathcal L_{\mathcal P,t_{0}})$ can be expressed in terms of $M(f,1)$ and $M(f,3)$ (and we were not able to find such an expression numerically).

Recall that the number $M(f,s)$ is defined by a partial integral $M(f,s)=\frac{(2\pi)^s}{\Gamma(s)}\int\limits_{\sqrt{N}^{-1}}^{\infty}f(iz)z^{s-1}dz$. We can consider more general integral
$$
M(f,s;t):=\frac{(2\pi)^s}{\Gamma(s)}\int\limits_{\sqrt{t}^{-1}}^{\infty}f(iz)z^{s-1}dz,
$$
where $t\in \langle 0,+\infty)$ is a non-negative number and $s$ belongs to the half-plane of convergence; in particular $M(f,s)=M(f,s;N)$.

Let $f:=f_{6/1}$, let $\mathcal{P}$ be the Picard-Fuchs operator for arrangement $253$ and let $t_0:=-2$. Since the Atkin-Lehner signs of the modular form $f_{192/2}$ equal
\[	\begin{array}{c|c c}
p&2&3\\
\hline
\text{sign}&-1&\phantom{-}1
	\end{array}
\]
and
\begin{eqnarray*}
	M(f,1)=M(f_{6/1},1;6)&=&0.07057956451083054732255139009133496203387610115943...\\
	M(f,3)=M(f_{6/1},3;6)&=&0.49691599739243476804039800218000443074993876412363...\\
	M(f_{6/1},1;\tfrac32)&=&0.00588018380632647168784781079042205384110122041895...\\
	M(f_{6/1},3;\tfrac32)&=&0.11354370318430276251965275943335072038142945788569...
\end{eqnarray*}
comparing generators of $\operatorname{Im}(\mathcal{L}^0_{\mathcal{P},t_0})$ (cf. Table \ref{tab:monodromy}) with $\sqrt{2}M(f_{192/2},s;\tfrac32)$
we get
\begin{align*}
	6.26847094349121003359079492495& \approx20\sqrt2M(f,1;6)+\frac{60\sqrt2}{\pi^2}M(f,3;6)\\
	7.96011334055139325749281017005&\approx48\sqrt2M(f,1;6)+\frac{48\sqrt2}{\pi^2}M(f,3;6)+64\sqrt2M(f,1;\tfrac{3}{2})-\frac{48\sqrt2}{\pi^{2}}M(f,3;\tfrac{3}{2})\\
\end{align*}
\noindent This suggests that in general to identify additional periods one has to consider partial integrals $M(f,s;t)$ for different values of the parameter $t$.

However, even this approach did not yield results for the remaining operators 244 and 274, since (unlike in all other cases) the rank 2 part $\operatorname{Im}(\mathcal{L}^0_{\mathcal{P},t_0})$ of $\mathcal{L}^0_{\mathcal{P},t_0}$ consist of \textit{imaginary} integrals and contains $\tfrac{L(f,2)}{2\pi i}$. 
Perhaps a different way of decomposing imaginary integrals, and consequently $L(f,2)$, is necessary.

\end{document}